\documentclass[a4paper]{amsart}

\usepackage{amsmath}
\usepackage{url}
\usepackage{amsthm}
\usepackage{amsfonts}
\usepackage{mathrsfs}
\usepackage{color}
\usepackage{enumerate}
\usepackage{amssymb}
\usepackage{tikz} % Diagram tool
\usetikzlibrary{calc}

\newtheorem{theorem}{Theorem}[section]

\newtheorem{lemma}[theorem]{Lemma}
\newtheorem{proposition}[theorem]{Proposition}
\newtheorem{corollary}[theorem]{Corollary}

\theoremstyle{definition}
\newtheorem{definition}[theorem]{Definition}

\theoremstyle{remark}
\newtheorem{remark}[theorem]{Remark}

\numberwithin{equation}{section}

\newcommand{\R}{\mathbb{R}}

\newcommand{\ip}[2]{\left\langle #1, #2\right\rangle}
\newcommand{\pf}[2]{\frac{\partial #1}{\partial #2}}

\renewcommand{\div}{\textup{div}}

\newcommand{\abs}[1]{\left|#1\right|}

\numberwithin{equation}{section}

\newcommand{\sL}{{\mathcal L}}

\begin{document}
\title[]{Uniqueness Theorems of Self-Conformal Solutions To Inverse Curvature Flows}
\author[Nicholas Chin]{Nicholas Cheng-Hoong Chin}
\author[Frederick Fong]{Frederick Tsz-Ho Fong}
\author[Jingbo Wan]{Jingbo Wan}
\address{Department of Mathematics, Hong Kong University of Science and Technology, Clear Water Bay, Kowloon, Hong Kong}

\email{nchchin@connect.ust.hk}
\email{frederick.fong@ust.hk}
\email{jwanac@connect.ust.hk}

\date{Submitted on 18 December, 2019, Revised on 18 April, 2020}

\maketitle
\begin{abstract}
It has been known in \cite{DLW, KLP, CCF} that round spheres are the only closed homothetic self-similar solutions to the inverse mean curvature flow and parabolic curvature flows by degree $-1$ homogeneous functions of principal curvatures in the Euclidean space.

In this article, we prove that the round sphere is rigid in stronger sense: under some natural conditions such as star-shapedness, round spheres are the only closed solutions to the above-mentioned flows which evolve by diffeomorphisms generated by conformal Killing fields.
\end{abstract}

\section{Introduction}
In this article, we study uniqueness problems of \emph{self-conformal} solutions to inverse curvature flows including the inverse mean curvature flow (IMCF). They are solutions which evolve by diffeomorphisms generated by conformal Killing fields. The flows that we consider are parabolic flows on Euclidean hypersurfaces $\Sigma^{n\geq2} \subset \R^{n+1}$ of the form:
\begin{equation}
\label{eq:FlowInverse}
\left(\pf{F}{t}\right)^\perp = -\frac{1}{\rho}\nu.	
\end{equation}
Here $\nu$ is the evolving Gauss map, and $\rho(\lambda_1,\cdots,\lambda_n) : \Gamma \subset \R^n \to \R_+$ is a positive, homogeneous of degree $1$, symmetric and $C^2$ function of principal curvatures defined on a certain open cone\footnote{Here a cone $\Gamma$ means subset of $\R^n$ such that $x \in \Gamma$ implies $tx \in \Gamma$ for any $t \geq 0$.} $\Gamma \subset \R^n$ such that $\pf{\rho}{\lambda_i} > 0$ on $\Gamma$. For instance, when $\rho = H$ (the mean curvature) and take $\Gamma = \R^m$, the flow \eqref{eq:FlowInverse} is the well-known inverse mean curvature flow (IMCF) which is the major tool of proving many geometric inequalities such as the Riemannian Penrose inequality by Huisken-Ilmanen \cite{HI01}, and Minkowski's and Alexandrov-Fenchel's inequalities by Guan-Li \cite{GL}, Brendle-Hung-Wang \cite{BHW}, Wei \cite{We}, and many others.

We call a solution to the flow to be \emph{self-conformal} if it evolves by diffeomorphisms generated by a conformal Killing field, which is a vector field $V$ on $\R^{n+1}$ that satisfies $\mathcal{L}_V \delta = \frac{2\div(V)}{n+1}\delta$ where $\delta$ is the standard Euclidean metric on $\R^{n+1}$. Self-conformal solutions include homothetic self-similar solutions and translating solitons as special cases. In particular, when $V = cX$ where $c$ is a constant and $X$ is the position vector field, then $\mathcal{L}_V\delta = c\delta$, and the diffeomorphisms generated by $V$ are rescaling maps. Solutions that evolve by these rescaling diffeomorphisms are called \emph{homothetic self-similar} solutions. When $V$ is a constant vector, the diffeomorphisms generated by $V$ are translations and self-conformal solutions along this $V$ are known as \emph{translating solitons}.

We focus only on compact Euclidean hypersurfaces in this article. For compact homothetic self-similar solutions, it has been proven that round spheres are the only closed homothetic self-similar solutions to \eqref{eq:FlowInverse} by G. Drugan, H. Lee and G. Wheeler \cite{DLW} in the case $\rho = H$ (i.e. IMCF); by K.K. Kwong, H. Lee and J. Pyo \cite{KLP} in the case $\rho = (\sigma_i/\sigma_j)^{1/(i-j)}$ and their positive linear combinations. Here $\sigma_i = \sum_{k_1 < \cdots < k_i} \lambda_{k_1} \cdots \lambda_{k_i}$ is the $i$-th symmetric polynomial of principal curvatures. More generally in \cite{CCF}, A. Chow, K.W. Chow and the second-named author proved round spheres are the only homothetic self-similar solutions to \eqref{eq:FlowInverse} for any general symmetric homogeneous $C^1$ function $\rho$ of degree 1 such that \eqref{eq:FlowInverse} is parabolic. In both \cite{DLW} and \cite{KLP}, the Hsiung-Minkowski's identities \cite{Hs} and their weighted variants \cite{KLP}, which are integral formulae involving $\sigma_k$'s, were used to show that such a homothetic self-similar solution must be umbilic. The general case considered in \cite{CCF} was proved by a different argument, by showing that such a closed homothetic self-similar solution must be rotationally symmetric about any axis through the origin. Consequently, such a solution must be a round sphere. On the other hand, homothetic self-similar solutions which are non-compact or with higher codimensions are much less rigid, as there are multiple examples constructed in \cite{HI97,CL,DLW,CCF,Hui}, some of which even have the same topological type but are geometrically distinct.

The main purpose of this article is to show the round spheres are rigid in stronger sense than in \cite{DLW, KLP, CCF}, that they are the only closed \emph{self-conformal} solutions to the inverse mean curvature flow (and some other non-linear flows by homogenenous speed functions) under some natural assumptions such as star-shapedness. Here is the summary of our main results:

\vskip 0.2cm
\label{main_thm}
\noindent\textbf{Main Theorem}. Suppose $\Sigma^n \subset \R^{n+1}$ is a closed self-conformal solution to the flow $(\partial_t F)^\perp = \varphi \nu$ along a conformal Killing vector field $V$ (see Definition \ref{def:ConformalKilling}). Then, if any of the following conditions below is met, then $\Sigma^n$ must be a round sphere.
\begin{enumerate}[(i)]
	\item $n = 2$, and $\varphi = -\dfrac{1}{H}$; or
	\item $n \geq 3$, $\varphi = -\dfrac{1}{H}$, and $\div(K)$ is constant on $\Sigma$; or
	\item $n \geq 2$, $\Sigma^n$ is star-shaped, and $\varphi = -\dfrac{1}{\rho}$ where $\rho$ is a symmetric homogeneous function of degree $1$ in the class $\mathcal{C}$ defined in p.\pageref{classC}; or
	\item $n \geq 2$, $\varphi = -\frac{\sigma_{k-1}}{\sigma_k}$, and $\Sigma^n$ and $V$ satisfy the condition at some time $t \geq 0$:
	\[\frac{\int_{\Sigma_t} \sigma_k \div(V)\,d\mu}{\int_{\Sigma_t}\sigma_k\,d\mu} = \frac{\int_{\Sigma_t} \sigma_{k-1} \div(V)\,d\mu}{\int_{\Sigma_t}\sigma_{k-1}\,d\mu}.\]
\end{enumerate}

All four cases (i)-(iv) include the inverse mean curvature flow (with $\rho = H$ for (iii), and $k = 1$ for (iv)). For (i) and (ii), the uniqueness result is proved using the Willmore energy which is monotone along IMCF in any dimension $\geq 2$, and is a conformal invariant in dimension $2$. No star-shaped condition is needed in these cases. For (iii), the star-shaped condition enables us to apply the estimates by Gerhardt \cite{G} and Urbas \cite{U} to prove our uniqueness result. For (iv), we consider the monotone quantities considered by P. Guan and J. Li in \cite{GL}, and prove that they are stationary along the flow. Cases that satisfy the assumption in (iv) are discussed in several remarks after Theorem \ref{thm:reflect}. We will prove (i) and (ii) in Theorem \ref{thm:2D}, (iii) in Theorem \ref{thm:StarShaped}, and (iv) in Theorem \ref{thm:reflect}.

\vskip 0.2cm
\noindent\textbf{Acknowledgement.}
The authors would like to thank the referee for his/her valuable and insightful comments on our previous version of the article. We also thank him/her for suggestions of improving the expositions of the article. Nicholas Chin is partially supported by the HKUST postgraduate studentship. Jingbo Wan is partially supported by the HKUST Undergraduate Research Opportunity Project (UROP). The research conducted is partially supported by the second-named author's Hong Kong RGC Early Career Grant \#26301316 and General Research Fund \#16302417.

\section{Preliminaries}

\begin{definition}[Conformal Killing fields]
\label{def:ConformalKilling}
A vector field $V$ on a Riemannian manifold $(M,g)$ is said to be a \emph{conformal Killing field} if there exists a smooth function $\alpha_V : M \to \R$ such that
\begin{equation}
\label{eq:ConformalKilling}
\mathcal{L}_V g = 2\alpha_V g	
\end{equation}
where $\mathcal{L}_V$ denotes the Lie derivative along $V$.
\end{definition}

\begin{remark}
Express $V = V^i \pf{}{x_i}$ in local coordinates, then we have $(\mathcal{L}_V g)_{ij} = \nabla_i V_j + \nabla_j V_i$, and so it is necessary that $\alpha_V = \frac{\div_g(V)}{\dim M}$ for $V$ to be a conformal Killing field.
\end{remark}

Suppose $V$ is a conformal Killing field in $(M,g)$, then by considering the definition of Lie derivatives, we have
\[\frac{d}{dt}\left(\Phi^V_t\right)^*g = \left(\Phi^V_t\right)^*(\mathcal{L}_V g) = \left(\Phi^V_t\right)^*(2\alpha_V g),\]
where $\Phi^V_t : \R^{n+1} \to \R^{n+1}$ is the diffeomorphism family generated by $V$. One can easily solve this ODE and show that $\Phi^V_t$ pulls back $g$ by
\begin{equation}
\label{eq:ConformalFlow}
\left(\Phi^V_t\right)^*g = e^{2\int_0^t \left(\Phi^V_\tau\right)^*\alpha_V d\tau}g.
\end{equation}
Hence, $\Phi^V_t$ is a conformal map on the set of points in $M$ on which $\Phi^V_t$ is well-defined.

On Euclidean spaces $\R^{n+1\geq 3}$, the set of conformal Killing fields is completely known. They are of the form:
\begin{align*}
    V(X)=v+AX+\mu X+2\ip{b}{X}X-|X|^2b
\end{align*}
for some $v,b\in\R^{n+1}$, $\mu\in\R$ and $A\in O(n+1)$. Consequently, $\div(V)$ is an affine linear function on $\R^{n+1}$ (see Proposition \ref{prop:alpha_affine}). Each term in $V$ corresponds to a specific kind of conformal transformations. The constant vector $v$ gives the translation, $X \mapsto AX$ gives a rotations or reflection, $X \to \mu X$ gives a homothetic rescaling, and $X \mapsto 2\ip{b}{X}X-|X|^2b$ corresponds to inversions.

\begin{definition}[Self-Conformal Solutions]
\label{def:SelfConformal}
A complete hypersurface $\Sigma^n \subset \R^{n+1}$ is said to be a \emph{self-conformal solution} to the flow \eqref{eq:FlowInverse} if the flow \eqref{eq:FlowInverse} initiating from $\Sigma$ evolves by
\[\Sigma_t = \Phi^V_t \circ \Sigma\]
for some conformal Killing field $V$ in the Euclidean space $\R^{n+1}$.
\end{definition}

When $V = \mu X$ where $\mu$ is a non-zero constant, and $X$ is the position vector field in $\R^{n+1}$, then the diffeomorphism family generated by $V$ is given by $\Phi^V_t(X) = e^{\mu t}X$. A self-conformal solution with respect to this $V$ is called a \emph{homothetic self-similar solution}. Such a solution is called a \emph{self-expander} if $\mu > 0$; and a \emph{self-shrinker} if $\mu < 0$. When $V$ is a constant vector field, then $\Phi^V_t(p) = p + tV$, so a self-conformal solution with respect to this $V$ is called a \emph{translating soliton}, or simply a \emph{translator}. Therefore, the class of self-conformal solutions include many well-studied special solutions of curvature flows.

As mentioned in the introduction, the following results are known about the uniqueness of closed homothetic self-similar solutions to IMCF and many other flows as well:

\begin{theorem}[Drugan-Lee-Wheeler \cite{DLW}, Kwong-Lee-Pyo \cite{KLP}, Chow-Chow-Fong \cite{CCF}]
The only closed homothetic self-similar solutions to the flow $(\partial_t F)^\perp = -\frac{1}{\rho}\nu$, where $\rho$ satisfies conditions (i)-(iv) stated in p.\pageref{classC} (condition (v) is not needed), are round spheres.
\end{theorem}

Therefore, the main theorem stated in p.\pageref{main_thm} strengthens the above rigidity result by showing that round spheres are unique even in the class of self-conformal solutions of IMCF and some types of flows satisfying conditions (i)-(iv) in p.\pageref{classC}.

\section{Uniqueness Theorem of Self-Conformal Surfaces in $\R^3$}
Using the Willmore energy, the round $2$-spheres in $\R^3$ can be shown to be the only closed self-conformal solutions to the inverse mean curvature flow (IMCF):
\[\left(\pf{F}{t}\right)^\perp = -\dfrac{1}{H}\nu\]
where $H$ is the mean curvature. Furthermore, the round $n$-sphere in $\R^{n+1}$, where $n \geq 3$, can be similarly shown to be the only solution to the IMCF which evolves along a conformal Killing field .

Throughout the article, we will implicitly assume that solutions to the flows satisfy curvature conditions so that the speed function is well-defined. For instance, for the inverse mean curvature flow $(\partial_t F)^\perp = -\frac{1}{H}\nu$, we assume the hypersurface is mean-convex; while for the flow $(\partial_t F)^\perp = -\frac{\sigma_{k-1}}{\sigma_k}\nu$, we assume the hypersurface is $\sigma_k$-convex.

\begin{theorem}
\label{thm:2D}
The only closed self-conformal solutions $\Sigma^2 \subset \R^3$ to the inverse mean curvature flow $\big(\partial_t F\big)^\perp = -\frac{1}{H}\nu$ are round spheres. Furthermore, the only closed self-conformal solutions in $\R^{n+1}$, where $n \geq 3$, to the inverse mean curvature flow associated with a conformal Killing vector field $V$ in $\R^{n+1}$ with constant $\div(V)$ are round spheres.
\end{theorem}

\begin{proof}
The Willmore energy
\[W(\Sigma) := \int_\Sigma H^n\,d\mu_\Sigma\]
is well-known (see e.g. \cite{Ch,Wh}) to be a conformal invariant for surfaces (i.e. $n = 2$) in $\R^3$, in a sense that if $\Phi : \R^{n+1} \to \R^{n+1}$ is a conformal diffeomorphism, then $W(\Phi(\Sigma)) = W(\Sigma)$. Generally for $n \geq 3$, $W$ is invariant when $\Phi$ is an isometry. We are going to show that $W(\Sigma_t)$ is monotone decreasing along the IMCF in all dimensions. Note that by reparametrization one can replace the embedding $F$ by one that satisfies $\pf{F}{t} = - \frac{1}{H}\nu$, and the Willmore energy defined as an integral is unchanged under reparametrizations. 

The Willmore energy $W(\Sigma_t)$ evolves under a general variation $\pf{F}{t} = f\nu$ by:
\begin{align*}
\frac{d}{dt}W(\Sigma_t) & = \int_\Sigma nH^{n-1}\left(\Delta f + f \abs{A}^2\right) + H^n(-fH) \,d\mu\\
& = \int_\Sigma -n(n-1)H^{n-2}\langle \nabla H, \nabla f\rangle + nfH^{n-1}\left(\abs{A}^2 - \frac{H^2}{n}\right)\,d\mu.
\end{align*}

Now take $f = -\dfrac{1}{H}$, we get
\begin{align*}
\frac{d}{dt}W(\Sigma_t) & = \int_\Sigma -n(n-1)H^{n-4} \abs{\nabla H}^2 - nH^{n-2}\left(\abs{A}^2 - \frac{H^2}{n}\right)\,d\mu \leq 0
\end{align*}
whenever $H > 0$ (which is implicitly assumed for the flow). The equality holds if and only if $H$ is a constant and $\Sigma_t$ is umbilic.

Now when $n = 2$, if $\Sigma_t$ is a self-conformal solution to the IMCF, then $\frac{d}{dt} W(\Sigma_t) \equiv 0$ for any $t > 0$, and consequently $H$ is constant and $\Sigma$ is umbilic, and hence is a round sphere in $\R^3$. Generally in higher dimensions, although $W$ is not a conformal invariance, it is invariant under isometry and rescaling. Hence, when the associated conformal Killing vector field $V$ has constant $\div(V)$, $\Sigma_t$ evolves by compositions of rescalings and isometries (including rotations and translations). We have $\frac{d}{dt}W(\Sigma_t) \equiv 0$ for any $t > 0$, and so we conclude that $\Sigma$ is a round sphere.
\end{proof}

\begin{remark}
The case $n \geq 3$ in Theorem \ref{thm:2D}	extends the uniqueness result in \cite{DLW} from homothetic self-similar solutions to wider class of self-conformal solutions which could evolve by rescaling, rotations, translations, and compositions of them.
\end{remark}

\section{Uniqueness Theorems in Higher Dimensions}
To extend Theorem \ref{thm:2D} to higher dimensional IMCF and to other flows \eqref{eq:FlowInverse}, we seek other approaches. One approach is the use of a point-wise conformal invariant $E_{ij}(a)$ to be defined below, and also the asymptotic properties of the flow \eqref{eq:FlowInverse} proved by Gerhardt \cite{G} and Urbas \cite{U}. The assumption on star-shapedness will be needed. 

A hypersurface $\Sigma^n\subset\R^{n+1}$ is \emph{star-shaped} if it can be written as a smooth radial graph over the unit sphere $\mathbb{S}^n$; that is, there exists a smooth function $u:\mathbb{S}^n\to(0,\infty)$ such that 
\begin{align*}
    \Sigma=\left\{u(p)p:p\in\mathbb{S}^n\right\}.
\end{align*}

Homothetic self-similar solutions to the flow \eqref{eq:FlowInverse}, whose conformal Killing fields are constant multiples of the position vector $X$, must be star-shaped. It is because $-\frac{1}{\rho} = \langle \mu X, \nu\rangle$, and hence $\langle X, \nu\rangle$ is non-vanishing on $\Sigma$. For self-conformal solutions, the star-shapedness is regarded as an additional condition.

Another approach of extending Theorem \ref{thm:2D} is to consider the monotone quantities
\[Q_k(t) :=\frac{\left(\int_{\Sigma_t}\sigma_kd\mu\right)^{\frac{1}{n-k}}}{\left(\int_{\Sigma_t}\sigma_{k-1}d\mu\right)^{\frac{1}{n-k+1}}}\]
where $\sigma_k$ is the $k$-th elementary symmetric function of principal curvatures with $\sigma_k(1,\cdots,1) = {n \choose k}$. These quantities are proved by Guan-Li in \cite{GL} to be monotone along the flow $(\partial_t F)^\perp = -\frac{\sigma_{k-1}}{\sigma_k}\nu$, and are stationary if and only if $\Sigma_t$ is a round sphere. Using these quantities, the Hsiung-Minkowski's identities, one can derive an identity which is sufficient to conclude the self-conformal solution is a round sphere.

 \subsection{A family of conformal invariant $2$-tensors} We consider the following family of $2$-tensors on $\Sigma$ which are invariant under conformal diffeomorphisms. They will be used, combining with the curvature asymptotics of the flow by \cite{G,U} to show the uniqueness of star-shaped self-conformal solutions.

\begin{lemma}
\label{lma:ConformalInvariant}
Each of the $2$-tensors in the following $1$-parameter family on a Euclidean hypersurface $\Sigma$
\[E_{ij}(a) := Hh_{ij}+aH^2g_{ij}-\frac{n}{2}h_i^k h_{kj} -\frac{2an+1}{2}|A|^2g_{ij}, \text{ where } 2an+1\geq 0,\]
is invariant under conformal diffeomorphisms $\Phi : \R^{n+1} \to \R^{n+1}$, and each of them is zero at a point $p \in \Sigma$ if and only if $\Sigma$ is umbilic at $p \in \Sigma$. Therefore, $E(a) \equiv 0$ on $\Sigma$ if and only if $\Sigma$ is a round sphere.
\end{lemma}

\begin{proof}
Let $\Phi : \R^{n+1} \to \R^{n+1}$ be a conformal diffeomorphism, and $\widetilde{\Sigma} := \Phi(\Sigma)$. Denote geometric quantities of $\Sigma$ by $g_{ij}$, $h_{ij}$, etc., and those of $\widetilde{\Sigma}$ by $\widetilde{g}_{ij}$, $\widetilde{h}_{ij}$, etc.

Under a conformal diffeomorphism, the first and second fundamental forms of $\Sigma$ and $\widetilde{\Sigma}$ are related by
\begin{align*}
\widetilde{g}_{ij} & =e^{2f}g_{ij},\\
\widetilde{h}_{ij} & =e^{f}\left(h_{ij}-(D_{\nu}f) g_{ij}\right).
\end{align*}

Using these, one can compute that curvatures of $\Sigma$ and $\widetilde{\Sigma}$ are related by:
\begin{align*}
\widetilde{H} & =\widetilde{g}^{ij}\widetilde{h}_{ij}=e^{-f}(H-nD_{\nu}f)\\
(\widetilde{A^2})_{ij} & =\widetilde{h}_{ip}\widetilde{g}^{pq}\widetilde{h}_{qj}=(h_{i}^q-(D_{\nu}f) \delta_i^q )(h_{qj}-(D_{\nu}f) g_{qj})\\
& =(A^2)_{ij}-2(D_{\nu}f) h_{ij}+(D_{\nu}f)^2g_{ij}\\
|\widetilde{A}|^2 & =\widetilde{g}^{ij}(\widetilde{A^2})_{ij}=e^{-2f}(|A|^2-2H(D_{\nu}f)+n(D_{\nu}f)^2)
\end{align*}

Using these, we can directly verify each $E_{ij}(a)$ is a conformal invariance, as all terms involving $D_\nu f$ got cancelled:

\begin{align*}
\widetilde{E}_{ij}(a) &= \widetilde{H}\widetilde{h}_{ij}+a\widetilde{H}^2\widetilde{g}_{ij}-\frac{n}{2}(\widetilde{A^2})_{ij}-\frac{2an+1}{2}|\widetilde{A}|^2\widetilde{g}_{ij}
 \\
  &= e^{-f}(H-nD_{\nu}f)e^{f}(h_{ij}-(D_{\nu}f)g_{ij})\\
  &\ \ \ +ae^{-2f}(H-nD_{\nu}f)^2e^{2f}g_{ij}\\
  &\ \ \ -\frac{n}{2}[(A^2)_{ij}-2(D_{\nu}f) h_{ij}+(D_{\nu}f)^2g_{ij}]\\
  &\ \ \ -\frac{2an+1}{2}e^{-2f}(|A|^2-2HD_{\nu}f+n(D_{\nu}f)^2)e^{2f}g_{ij}\\
  &= Hh_{ij}+aH^2g_{ij}-\frac{n}{2}(A^2)_{ij}-\frac{2an+1}{2}|A|^2g_{ij} =E_{ij}(a).
\end{align*}

Next, we show that for $2an+1\geq 0$, $E_{ij}(a)=0$ on $\Sigma$ if and only if $\Sigma$ is a round sphere. Consider an orthonormal basis $\{e_i\}$ of $T_p \Sigma$ at a point $p \in \Sigma$ such that $h_{ij}= h(e_i, e_j) = \lambda_i \delta_{ij}$. Then, $E_{ij}(a)$ is also diagonal with eigenvalues given by 
\[H\lambda_i+aH^2-\frac{n}{2}\lambda_i^2-\frac{2an+1}{2}|A|^2=-\frac{n}{2}\left(\lambda_i-\frac{H}{n}\right)^2-\frac{2an+1}{2}\abs{A^\circ}^2\]
where $A^\circ := A - \frac{H}{n}g$.

We see that $E_{ij}(a)=0$ at $p$ if and only if 
\[-\frac{n}{2}\left(\lambda_i-\frac{H}{n}\right)^2-\frac{2an+1}{2}|A^\circ|^2=0\]
for all $i = 1, \cdots, n$. Since $2an+1\geq 0$, this condition is equivalent to $\lambda_i=\frac{H}{n}$ for any $i = 1, \cdots, n$.  This shows by picking $2an + 1 \geq 0$, one has $E_{ij}(a) \equiv 0$ on $\Sigma$ if and only if $\Sigma$ is totally umbilic (i.e. $\Sigma$ is a round sphere in $\R^{n+1}$).
\end{proof}

Using the conformal invariants $E_{ij}$, we can show that round spheres are the only closed, star-shaped self-conformal solutions to a large class of inverse curvature flows by homogeneous speed functions $-\frac{1}{\rho}$ considered by Gerhardt \cite{G} and Urbas \cite{U}. We denote $\mathcal{C}$ to be the class of functions $\rho$ of principal curvatures considered in \cite{G} and \cite{U}. Precisely, $\rho$ is in the class $\mathcal{C}$ if and only if all of the following hold:\label{classC}
\begin{enumerate}[(i)]
	\item $\rho$ is $C^2$ and is positive on an open cone $\Gamma \subset \R^n$ containing $(1,\cdots,1)$.
	\item $\rho$ is a symmetric function on $\Gamma$.
	\item $\rho$ is homogeneous of degree 1 on $\Gamma$.
	\item $\pf{\rho}{\lambda_i} > 0$ on $\Gamma$ for any $i$.
	\item $\left[\frac{\partial^2\rho}{\partial\lambda_i\partial\lambda_j}\right]$ is semi-negative definite on $\Gamma$.
\end{enumerate}
Examples of such $\rho$'s include $\rho = H$, $\rho = \frac{\sigma_k}{\sigma_{k-1}}$, $\rho=(\sigma_i/\sigma_j)^{1/(i-j)}$, $\rho=\sigma_k^{1/k}$, etc.

\begin{theorem}
\label{thm:StarShaped}
The only closed, star-shaped self-conformal solution $\Sigma^n \subset \R^{n+1}$ to the flow $(\partial_t F)^\perp = -\frac{1}{\rho}\nu$ with $\rho \in \mathcal{C}$ are round spheres.	
\end{theorem}

\begin{proof}
Let $F_t$ be the evolving embedding of $\Sigma_t$ along the flow \eqref{eq:FlowInverse}. By \cite{G,U}, the star-shaped condition is preserved so that $\Sigma_t$ can be regarded as a evolving graph over the unit sphere $\mathbb{S}^n$ with the round metric $g_{\mathbb{S}^n}$. Let $\sigma_{ij}$ be the local components of $g_{\mathbb{S}^n}$. Denote the evolving graph function by $u(t) : \mathbb{S}^n \to \R_+$, then from \cite{G, U} we have the following estimates: there exist constants $\beta, C > 0$ such that
\begin{align*}
0 < \frac{1}{C} e^{t/\mu} \leq u(t) & \leq Ce^{t/\mu}  \;\; \text{ where } \mu = \rho(1, \cdots, 1),\\
\abs{\nabla u(t)}_{g_{\mathbb{S}^n}} & \leq C,\\
u(t) h_i^j(t) & = \delta_i^j + O(e^{-\beta t}).
\end{align*}
Here $\nabla$ is the Levi-Civita connection of $\mathbb{S}^n$. By rescaling the flow by $\widetilde{F}_t := e^{-t/\mu}F_t$ so that the graph function becomes $\widetilde{u}(t) := e^{-t/\mu}u(t)$, the above estimates can be written as:
\begin{align*}
\frac{1}{C} \leq \widetilde{u}(t) & \leq C\\
\abs{\nabla\widetilde{u}(t)}_{g_{\mathbb{S}^n}} & \leq Ce^{-t/\mu},\\
\widetilde{u}(t) \widetilde{h}_i^j(t) & = \delta_i^j + O(e^{-\beta t})
\end{align*}
The first fundamental form of $\widetilde{F}_t$ is given by:
\[\widetilde{g} = \widetilde{u}^2 g_{\mathbb{S}^n} + d\widetilde{u} \otimes d\widetilde{u} = \widetilde{u}^2 g_{\mathbb{S}^n} + O\big(e^{-2t/\mu}\big)\]
which is uniformly equivalent to $g(0)$. We can then derive that
\begin{align*}
\widetilde{E}_{ij}(a) & = \widetilde{H}\widetilde{h}_{ij}+a\widetilde{H}^2\widetilde{g}_{ij}-\frac{n}{2}(\widetilde{A^2})_{ij}-\frac{2an+1}{2}|\widetilde{A}|^2\widetilde{g}_{ij}\\
& = \left(\frac{n}{\widetilde{u}} + O\big(e^{-\beta t}\big)\right)\left(\widetilde{u}^2\sigma_{ik}+O\big(e^{-2t/\mu}\big)\right)\left(\frac{1}{\widetilde{u}}\delta^k_j + O\big(e^{-\beta t}\big)\right)\\
& \hskip 0.4cm + a\left(\frac{n}{\widetilde{u}}+O\big(e^{-\beta t}\big)\right)^2\left(\widetilde{u}^2\sigma_{ij} + O\big(e^{-2t/\mu}\big)\right)\\
& \hskip 0.4cm -\frac{n}{2}(\widetilde{u}^2\sigma_{ik}+O\big(e^{-2t/\mu}\big))\left(\frac{1}{\widetilde{u}}\delta^k_l + O\big(e^{-\beta t}\big)\right)\left(\frac{1}{\widetilde{u}}\delta^l_j + O\big(e^{-\beta t}\big)\right)\\
& \hskip 0.4cm -\frac{2an+1}{2}\left(\frac{1}{\widetilde{u}}\delta_k^l + O\big(e^{-\beta t}\big)\right)\left(\frac{1}{\widetilde{u}}\delta_l^k + O\big(e^{-\beta t}\big)\right)\left(\widetilde{u}^2\sigma_{ik}+O\big(e^{-2t/\mu}\big)\right)\\
& = n\sigma_{ij} + an^2\sigma_{ij} - \frac{n}{2}\sigma_{ij}-\frac{2an+1}{2}n\sigma_{ij} + O\big(e^{-\min\{\beta, \mu^{-1}\}t}\big)\\
& = O\big(e^{-\min\{\beta, \mu^{-1}\}t}\big).
\end{align*}
As $t \to +\infty$, we have $\widetilde{E}_{ij}(a) \to 0$.

Now given that $\Sigma^n$ is a self-conformal solution so that $\widetilde{E}_{ij}(a)$ is in fact independent of $t$ by Lemma \ref{lma:ConformalInvariant}, we have $\widetilde{E}_{ij}(a) \equiv 0$ on $\Sigma$. Pick any $a$ such that $2an+1 \geq 0$, then by Lemma \ref{lma:ConformalInvariant} again we have proved that $\Sigma^n$ is a round sphere.

\end{proof}

\subsection{Conformal Killing fields in $\R^{n+1}$}
The uniqueness result in Theorem \ref{thm:StarShaped} holds for a very large class of inverse curvature flows, but it requires the hypersurface to be star-shaped. We next discuss the uniqueness result for a specific type of flows (still including IMCF) without the star-shapedness condition but with conditions on the associated conformal Killing fields instead.

Let's first discuss the classification of conformal Killing fields on Euclidean spaces, which is well-known. For completeness and reader's convenience, we include some relevant parts of the classification result here.

Denote the standard Euclidean metric on $\R^{n+1}$ by $\delta$, the Levi-Civita connection of $(\R^{n+1},\delta)$ by $D$, the Laplacian by $\bar{\Delta}$, and the standard coordinates by $x=(x^1,\ldots,x^{n+1})$. 

\begin{proposition}[c.f. \cite{CFT}]
\label{prop:alpha_affine}
Let $V=V^i\partial_i$ be a conformal Killing field on $\R^{n+1}$ where $n \geq 2$. Then, $\alpha=\frac{\div(V)}{n+1}$ is an affine linear function on $\R^{n+1}$, i.e.
\begin{align*}
    \alpha(x_1,\ldots,x_{n+1})=A+\sum_{i=1}^{n+1} B^i x_i
\end{align*}
for some constants $A, B^i \in\R$.
\end{proposition}

\begin{proof}
From $\sL_V\delta=2\alpha\delta$, we have $D_jV^k+D_kV^j=2\alpha\delta_{jk}$. Differentiating both sides with respect to $x^i$, we have 
\begin{align}
    \label{conformal factor eqn 1}\tag{*}
    D_iD_jV^k+D_iD_kV^j=2\delta_{jk}D_i\alpha
\end{align}
Permuting the indices $i,j,k$, we have the similar
\begin{align}
    \label{conformal factor eqn 2}\tag{**}
    D_jD_iV^k+D_jD_kV^i=2\delta_{ik}D_j\alpha \\
    \label{conformal factor eqn 3}\tag{***}
    D_kD_iV^j+D_kD_jV^i=2\delta_{ij}D_k\alpha
\end{align}
Then, taking the trace on (\ref{conformal factor eqn 1})$+$(\ref{conformal factor eqn 2})$-$(\ref{conformal factor eqn 3}) gives
\begin{align}
    \label{second derivative of V}
    D_iD_jV^k=\delta_{jk}D_i\alpha+\delta_{ik}D_j\alpha-\delta_{ij}D_k\alpha
\end{align}
Taking trace of (\ref{second derivative of V}), we have 
\begin{align*}
    \bar{\Delta}V^k&=\sum_i D_iD_iV^k \\
    &=(1-n)D_k\alpha.
\end{align*}
Using the commutativity of $D$ again, we have 
\begin{align*}
    \bar{\Delta}(D_iV^j)=D_i\left(\bar{\Delta}V^j\right)=(1-n)D_iD_j\alpha
\end{align*}
and similarly,
\begin{align*}
    \bar{\Delta}(D_jV^i)=(1-n)D_jD_i\alpha=(1-n)D_iD_j\alpha.
\end{align*}
Thus, 
\begin{equation}
\label{conformal factor}	
\bar{\Delta}(\mathcal{L}_Vg)_{ij} = \bar{\Delta}(D_iV^j)+\bar{\Delta}(D_jV^i) \; \implies \; 2\delta_{ij}\bar{\Delta}\alpha = 2(1-n)D_i D_j \alpha.
\end{equation}
Taking the trace, we get
\[2(n+1)\bar{\Delta}\alpha = 2(1-n)\bar{\Delta}\alpha,\]
proving that $\alpha$ is harmonic on $\R^{n+1}$. By \eqref{conformal factor}, we conclude $D_i D_j \alpha = 0$, completing the proof.
\end{proof}

By \eqref{second derivative of V}, we then conclude that each component $V^k$ has constant second derivatives, so we have:
\begin{corollary}
\label{cor:V is second order}
Let $V=V^i\partial_i$ be a conformal Killing field on $(\R^{n+1},\delta)$ where $n+1\geq 3$. Then each component of $V$ is quadratic; that is, for each $1\leq k\leq n+1$, 
\begin{align*}
    V^k(x^1,\ldots,x^{n+1})=a^k+b^k_lx^l+c^k_{ij}x^ix^j
\end{align*}
for some $a^k,b^k_l,c^k_{ij}\in\R$. 
\end{corollary}

\subsection{Uniqueness results via Guan-Li's monotone quantities.}

Recall that the uniqueness result in dimension two (Theorem \ref{thm:2D}) was proved by considering the Willmore energy which, in the surface case, is monotone along some inverse curvature flows including IMCF. In higher dimensions, the Willmore energy is not a monotone quantity under these flows. However, there are well-known monotone quantities $Q_k$ defined below along the flow $(\partial_t F)^\perp = -\frac{\sigma_{k-1}}{\sigma_k}\nu$ introduced by Guan-Li in \cite{GL}. These quantities are scale-invariant but not conformal-invariant. Combining the results in Proposition \ref{prop:alpha_affine} and Hsiung-Minkowski's identities, one can show they are stationary under conformal diffeomorphisms under a condition on the divergence of the conformal Killing vector field.

\begin{theorem}
\label{thm:reflect}
Let $\Sigma^n \subset \R^{n+1}$ be a closed self-conformal solution to the flow $(\partial_t F)^\perp = -\frac{\sigma_{k-1}}{\sigma_k}\nu$, where $k = 1, \cdots, n$, associated to a conformal Killing vector field $V$. Suppose for some $t \geq 0$,  $\div(V)$ satisfies the condition
\begin{equation}
\label{eq:ConditionV}
\frac{\int_{\Sigma_t} \sigma_k \div(V)\,d\mu}{\int_{\Sigma_t}\sigma_k\,d\mu} = \frac{\int_{\Sigma_t} \sigma_{k-1} \div(V)\,d\mu}{\int_{\Sigma_t}\sigma_{k-1}\,d\mu},
\end{equation}
then $\Sigma^n$ must be a round sphere.

In particular, if one of the following holds, then $\Sigma^n$ must be a round sphere.
\begin{enumerate}
\item when $\div(V)$ is a constant, or
\item when the ``center of mass'' of $\sigma_k$ is equal to that of $\sigma_{k-1}$ at some time $t \geq 0$. Precisely, it means that for any $i = 1, \cdots, n+1$, we have:
\begin{equation}
\label{eq:CenterOfMass}
\frac{\int_{\Sigma_t}\sigma_k x^i\,d\mu}{\int_{\Sigma_t}\sigma_k\,d\mu} = \frac{\int_{\Sigma_t}\sigma_{k-1} x^i\,d\mu}{\int_{\Sigma_t}\sigma_{k-1}\,d\mu}.
\end{equation}
Here $x^i$ is the $i$-th coordinate function in $\R^{n+1}$.
\end{enumerate}
\end{theorem}

Below are several interesting remarks of the theorem.

\begin{remark}
Homothetic self-similar solutions are special cases of self-conformal solutions with $V = \mu X$, where $\mu$ is a constant, and $X$ is the position vector field in $\R^{n+1}$. Such a vector field has a constant $\div(V)$ equal to $(n+1)\mu$. In this case, the result of Theorem \ref{thm:reflect} was proved by Drugan-Lee-Wheeler \cite{DLW} when $k = 1$, and Kwong-Lee-Pyo \cite{KLP} when $2 \leq k \leq n$. Hence, Theorem \ref{thm:reflect} further extends their uniqueness results to other conformal Killing vector fields $V$ with constant $\div(V)$. It is interesting that the proofs in \cite{DLW, KLP} and Theorem \ref{thm:reflect} all use the Hsiung-Minkowski's identities.\cite{Hs}
\end{remark}
\begin{remark}
Note also that when $k = 1$, the flow $(\partial_t F)^\perp = -\frac{\sigma_{k-1}}{\sigma_k}\nu$ is the inverse mean curvature flow. Hence, Theorem \ref{thm:reflect} gives an alternative proof of Theorem \ref{thm:2D} when $n \geq 3$.
\end{remark}
\begin{remark}
The ``center of mass'' condition \eqref{eq:CenterOfMass} is fulfilled when $\Sigma_t$ has reflectional symmetry across $(n+1)$ orthogonal planes in $\R^{n+1}$ that meet at a common point $P = (c_1, \cdots, c_{n+1}) \in \R^{n+1}$. Since in this case one has
\[\int_{\Sigma_t} \sigma_k (x_i - c_i)\,d\mu = \int_{\Sigma_t} \sigma_{k-1} (x_i - c_i)\,d\mu = 0 \; \text{ for any } i = 1, \cdots, n+1\]
\[\frac{\int_{\Sigma_t}\sigma_k x^i\,d\mu}{\int_{\Sigma_t}\sigma_k\,d\mu} = \frac{\int_{\Sigma_t}\sigma_k x^i\,d\mu}{\int_{\Sigma_t}\sigma_k\,d\mu} = c_i, \; \text{ for any } i = 1, \cdots, n+1.\]
Note that the uniqueness of self-conformal solutions under this reflectional symmetry assumption can also be obtained by invoking the classical Liouville's Theorem for conformal mappings in $\R^{n\geq 3}$, which asserts that any conformal mapping in $\R^{n\geq 3}$ is a composition of similarities (including rescalings, translations, rotations, and inversions), and the symmetry assumption reduces the self-conformal solution to a homothetic self-similar solution. Uniqueness of such a solution then follows from \cite{DLW} and \cite{KLP}. The proof of Theorem \ref{thm:reflect} in this special case does not rely on the classification results of conformal mappings in $\R^{n\geq3}$.
\end{remark}

\begin{proof}[Proof of Theorem \ref{thm:reflect}]
Let $V$ be a conformal Killing field in $\mathbb{R}^{n+1}$, then $\mathcal{L}_V \delta=\frac{2\div(V)}{n+1}\delta$. For each $k = 1, \cdots, n$, let $Q_k$ be the following scale-invariant quantity
\[Q_k(t) := \frac{\left(\int_{\Sigma_t}\sigma_k\,d\mu\right)^{\frac{1}{n-k}}}{\left(\int_{\Sigma_t}\sigma_{k-1}\,d\mu\right)^{\frac{1}{n-k+1}}}\]
which appeared in Guan-Li's work \cite{GL} about Alexandrov-Fenchel's inequalities.

Let $\Phi_t:\mathbb{R}^{n+1}\rightarrow \mathbb{R}^{n+1}$ be a diffeomorphsim generated by $V$, i.e. $\dot{\Phi}_t=V\circ\Phi_t$. A hypersurface $\Sigma$ deforms along the vector field $V$ if $\Sigma_t=\Phi_t(\Sigma_0)$ for any $t > 0$, and it implies $\left(\frac{\partial F}{\partial t}\right)^\perp = \ip{V}{\nu}\nu$. As in \cite[Lemma 5]{GL}, the integral $\int_{\Sigma_t} \sigma_l\,d\mu$ evolves by
\[\frac{d}{dt}\int_{\Sigma_t} \sigma_l\,d\mu=(l+1)\int_{\Sigma_t}\sigma_{l+1}\ip{V}{\nu}\,d\mu.\]
Using this, we can compute the evolution of $Q_k(t)$:
\begin{align*}
    \frac{d}{dt}Q_k(t)
    &= \frac{1}{n-k}\left(\int_{\Sigma_t} \sigma_k\,d\mu\right)^{\frac{1}{n-k}-1}\frac{\frac{d}{dt}\int_{\Sigma_t} \sigma_k\,d\mu}{\left(\int_{\Sigma_t}\sigma_{k-1}\,d\mu\right)^{\frac{1}{n-k+1}}}\\
    & \hskip 0.4cm -\frac{1}{n}\left(\int_{\Sigma_t} \sigma_k\,d\mu\right)^{\frac{1}{n-k}}\frac{\frac{d}{dt}\int_{\Sigma_t} \sigma_{k-1}\,d\mu}{\left(\int_{\Sigma_t}\sigma_{k-1}\,d\mu\right)^{\frac{1}{n-k+1}+1}}\\
    &= \frac{1}{n-k}\left(\int_{\Sigma_t} \sigma_k\,d\mu\right)^{\frac{1}{n-k}-1}\frac{(k+1)\int_{\Sigma_t} \sigma_{k+1}\ip{V}{\nu}\,d\mu}{\left(\int_{\Sigma_t}\sigma_{k-1}\,d\mu\right)^{\frac{1}{n-k+1}}}\\
    & \hskip 0.4cm -\frac{1}{n}\left(\int_{\Sigma_t} \sigma_k\,d\mu\right)^{\frac{1}{n-k}}\frac{k\int_{\Sigma_t} \sigma_{k}\ip{V}{\nu}\,d\mu}{\left(\int_{\Sigma_t}\sigma_{k-1}\,d\mu\right)^{\frac{1}{n-k+1}+1}}.
\end{align*}
Next we use the Hsiung-Minkowski's identities \cite{Hs} (see also  \cite{K}) for conformal Killing fields, which asserts that
\[ \int_{\Sigma} \ip{V}{\nu}\frac{\sigma_{k+1}}{{{n}\choose{k+1}}} \, d\mu + \int_{\Sigma} \frac{\div(V)}{n+1}\frac{\sigma_k}{{{n}\choose{k}}} \, d\mu = 0. \]
This shows

\begin{align*}
\frac{d}{dt} Q_k(t) &= -\frac{1}{n-k}\left(\int_{\Sigma_t} \sigma_k\,d\mu\right)^{\frac{1}{n-k}-1} \cdot \frac{\frac{k+1}{n+1}\cdot\frac{{{n}\choose{k+1}}}{{{n}\choose{k}}}\int_{\Sigma_t}\sigma_k\, \div(V)\,d\mu}{\left(\int_{\Sigma_t}\sigma_{k-1}\,d\mu\right)^{\frac{1}{n-k+1}}}\\
& \hskip 0.4cm +\frac{1}{n}\left(\int_{\Sigma_t} \sigma_k\,d\mu\right)^{\frac{1}{n-k}}\cdot\frac{\frac{k}{n+1}\frac{{{n}\choose{k}}}{{{n}\choose{k-1}}}\int_{\Sigma_t}\sigma_{k-1} \div(V)\,d\mu}{\left(\int_{\Sigma_t}\sigma_{k-1}\,d\mu\right)^{\frac{1}{n-k+1}+1}}\\
&= -\frac{Q_k(t)}{n+1}\left(\frac{\int_{\Sigma_t}\sigma_k\div(V)\,d\mu}{\int_{\Sigma_t}\sigma_k\,d\mu}-\frac{\int_{\Sigma_t}\sigma_{k-1} \div(V)\,d\mu}{\int_{\Sigma_t}\sigma_{k-1}\,d\mu} \right)\\
& = 0 & \text{ (by \eqref{eq:ConditionV})}
\end{align*}

From \cite{GL} the quantity $Q_k(t)$ is monotone decreasing along the flow $\partial_t F =-\frac{\sigma_k}{\sigma_{k-1}}\nu$, and $Q_k(t)$ is invariant under this flow if and only if $\Sigma_t$ is a round sphere. We conclude that $\Sigma$ is a round sphere. It completes the first part of the theorem.

The condition \eqref{eq:ConditionV} is clearly fulfilled when $\div(V)$ is a constant function. To see that \eqref{eq:ConditionV} is also fulfilled when \eqref{eq:CenterOfMass} holds, we recall from Proposition \ref{prop:alpha_affine} that $\div(V)$ is an affine linear function on $\R^{n+1}$. Write
\[\div(V) = A + \sum_{i = 1}^{n+1} B^i x_i.\]
Then, we have
\begin{align*}
\int_{\Sigma_t} \sigma_k \,\div(V)\,d\mu & = A\int_{\Sigma_t}\sigma_k\,d\mu + \sum_{i=1}^{n+1}B^i\int_{\Sigma_t}\sigma_k x_i\,d\mu\\
& = \left(A + \sum_{i=1}^{n+1}B^i\right)\int_{\Sigma_t}\sigma_k \,d\mu.
\end{align*}
Similarly, we have
\[\int_{\Sigma_t} \sigma_{k-1} \,\div(V)\,d\mu = \left(A + \sum_{i=1}^{n+1}B^i\right)\int_{\Sigma_t}\sigma_{k-1} \,d\mu.\]
It verifies the condition \eqref{eq:ConditionV}, and so by the first part of the theorem $\Sigma$ must be a round sphere.
\end{proof}

To conclude, we have demonstrated that round spheres are unique not only in the class of closed homothetic self-similar solutions of IMCF, but also in the class of closed self-conformal solutions for flows including IMCF. To the best of authors' knowledge, it is not known whether there are self-conformal solutions to IMCF which are not homothetic self-similar. For the compact case, Theorems \ref{thm:2D}, \ref{thm:StarShaped}, and \ref{thm:reflect} shows such an example must be in dimension three or higher, and it cannot be star-shaped or has constant $\div(V)$. For the non-compact case, the authors are not aware of any examples of self-conformal, but not homothetic self-similar, complete solutions to IMCF. It is an interesting problem to construct such a solution, or to show that it does not exist.

\appendix
\section{Inversion-Invariant Quantity and a Sharp Inequality}
The classical Liouville's Theorem for conformal maps (see e.g. \cite{Bl}) asserts that every conformal map of $\R^{n+1\geq 3}$ is a composition of similarities and an inversion. Inspired by this classification theorem and the proof of Theorem \ref{thm:reflect}, one may hope to find a conformal-invariant quantity that is monotone along, say, IMCF. The quantity
\[Q_1(\Sigma) = \frac{1}{|\Sigma|^{\frac{n-1}{n}}}\int_{\Sigma}Hd\mu\]
is only similarity-invariant. In view of Liouville's Theorem, it is natural to consider the modified $Q_1$ defined as
\[\overline{Q}(\Sigma) := Q_1(\Sigma) + Q_1(\widetilde{\Sigma}) = \frac{1}{|\Sigma|^{\frac{n-1}{n}}}\int_{\Sigma}Hd\mu + \frac{1}{|\widetilde{\Sigma}|^{\frac{n-1}{n}}}\int_{\widetilde{\Sigma}}\widetilde{H}d\widetilde{\mu}  \]
where $\widetilde{\Sigma}$ is the inversion of $\Sigma$ about the unit sphere centered at the origin.

It is clearly invariant under the inversion about unit sphere. One wishes that it is also monotone along IMCF, but it does not seem to be the case. However, we are able to prove a sharp geometric inequality concerning $\overline{Q}$ for star-shaped $\Sigma$.

\begin{theorem}
\label{thm:inequality}
Suppose $\Sigma^n \subset \R^{n+1 \geq 3}$ is a star-shaped closed hypersurface given as a radial graph of $f : \mathbb{S}^n \to \R_+$ over the unit sphere $\mathbb{S}^n$. Let $R:=\displaystyle\sup_{p\in\mathbb{S}^n}f(p)$ and $r:=\displaystyle\inf_{p\in\mathbb{S}^n}f(p)$. Then, we have 
\begin{align}
\label{eq:inequality}
    \left(\frac{r}{R}\right)^{\frac{3(n-1)}{2}}\frac{2n|\mathbb{S}^n|}{\left(|\Sigma|\cdot|\widetilde{\Sigma}|\right)^{\frac{n-1}{2n}}}
    %%%
    \leq \overline{Q}(\Sigma)
    %%%
    \leq\left(\frac{R}{r}\right)^{\frac{3(n-1)}{2}}\frac{2n|\mathbb{S}^n|}{\left(|\Sigma|\cdot|\widetilde{\Sigma}|\right)^{\frac{n-1}{2n}}}
\end{align}
where equalities hold if and only if $\Sigma$ is a round sphere.
\end{theorem}

\begin{proof}
Denote the standard round metric on $\mathbb{S}^n$ by $g_{\mathbb{S}^n}:=\iota^*\delta$, where $\delta$ is the flat metric on $\R^{n+1}$ and $\iota:\mathbb{S}^n\to\R^{n+1}$ is the inclusion. Let $F_{\mathbb{S}^n}$ be a local parametrization of $\mathbb{S}^n$ with local coordinates $(x^i)$. The components of $g_{\mathbb{S}^n}$ with respect to the local coordinates $(x^i)$ on $\mathbb{S}^n$ are denoted by $\sigma_{ij}$, and let $\nabla$ be the Levi-Civita connection on $\left(\mathbb{S}^n,g_{\mathbb{S}^n}\right)$.

Then, the star-shaped hypersurface $\Sigma$ and its inversion $\widetilde{\Sigma}$ can be locally parameterized by 
\[F:=fF_{\mathbb{S}^n} \;\; \text{ and } \;\; \widetilde{F} = \frac{1}{f}F_{\mathbb{S}^n}\]
respectively.

Denote the geometric quantities of $\Sigma$ by $g_{ij}$, $h_{ij}$, etc., and those of $\widetilde{\Sigma}$ by $\widetilde{g}_{ij}$, $\widetilde{h}_{ij}$, etc. By direct computations, one can get:
  	\begin{align*}
        \frac{\partial F}{\partial x^i} & =f\left[\left(\nabla_i\log f\right)F_{\mathbb{S}^n}+\frac{\partial F_{\mathbb{S}^n}}{\partial x^i}\right],\\
        g_{ij} & =f^2\left[\sigma_{ij}+\left(\nabla_i\log f\right)\left(\nabla_j\log f\right)\right],\\
        g^{ij} & =\frac{1}{f^2}\left[\sigma^{ij}-\frac{\left(\nabla^i\log f\right)\left(\nabla^j\log f\right)}{1+\left|\nabla\log f\right|^2}\right],\\
        d\mu &=f^n\sqrt{1+\left|\nabla\log f\right|^2}d\mu_{\mathbb{S}^n},\\
        \nu & =\frac{1}{\sqrt{1+\left|\nabla\log f\right|^2}}\left[F_{\mathbb{S}^n}-\left(\nabla^k\log f\right)\frac{\partial F_{\mathbb{S}^n}}{\partial x^k}\right],\\
        h_{ij} & =\frac{f}{\sqrt{1+\left|\nabla\log f\right|^2}}\left[\sigma_{ij}+\left(\nabla_i\log f\right)\left(\nabla_j\log f\right)-\nabla_i\nabla_j\log f\right],\\
        H &=\frac{1}{f\sqrt{1+\left|\nabla\log f\right|^2}}\left[n-\Delta\log f+\frac{\left(\nabla^i\log f\right)\left(\nabla^j\log f\right)\left(\nabla_i\nabla_j\log f\right)}{1+\left|\nabla\log f\right|^2}\right].
    \end{align*}
By replacing $f$ by $\frac{1}{f}$, we obtain:
    \begin{align*}
        \frac{\partial\widetilde{F}}{\partial x^i} & =\frac{1}{f}\left[-\left(\nabla_i\log f\right)F_{\mathbb{S}^n}+\frac{\partial F_{\mathbb{S}^n}}{\partial x^i}\right],\\
        \widetilde{g}_{ij} & =\frac{1}{f^2}\left[\sigma_{ij}+\left(\nabla_i\log f\right)\left(\nabla_j\log f\right)\right],\\
        \widetilde{g}^{ij} & =f^2\left[\sigma^{ij}-\frac{\left(\nabla^i\log f\right)\left(\nabla^j\log f\right)}{1+\left|\nabla\log f\right|^2}\right],\\
        d\widetilde{\mu} & =\frac{\sqrt{1+\left|\nabla\log f\right|^2}}{f^n}d\mu_{\mathbb{S}^n},\\
        \widetilde{\nu} & =\frac{1}{\sqrt{1+\left|\nabla\log f\right|^2}}\left[F_{\mathbb{S}^n}+\left(\nabla^k\log f\right)\frac{\partial F_{\mathbb{S}^n}}{\partial x^k}\right],\\
        \widetilde{h}_{ij} & =\frac{1}{f\sqrt{1+\left|\nabla\log f\right|^2}}\left[\sigma_{ij}+\left(\nabla_i\log f\right)\left(\nabla_j\log f\right)+\nabla_i\nabla_j\log f\right],\\
        \widetilde{H} & =\frac{f}{\sqrt{1+\left|\nabla\log f\right|^2}}\left[n+\Delta\log f-\frac{\left(\nabla^i\log f\right)\left(\nabla^j\log f\right)\left(\nabla_i\nabla_j\log f\right)}{1+\left|\nabla\log f\right|^2}\right].
    \end{align*}
We can then easily derive the relation between the mean curvatures $H$ and $\widetilde{H}$:
\begin{equation}
\label{eq:mean curvatures of inversion}
    \widetilde{H}=-f^2H+\displaystyle\frac{2nf}{\sqrt{1+\left|\nabla\log f\right|^2}}
\end{equation}

By \eqref{eq:mean curvatures of inversion}, we have
\begin{align*}
    Q_1(\widetilde{\Sigma}) = \frac{1}{|\widetilde{\Sigma}|^{\frac{n-1}{n}}}\int_{\widetilde{\Sigma}}\widetilde{H}d\widetilde{\mu} = \frac{-\displaystyle\int_{\Sigma}\frac{H}{f^{2(n-1)}}d\mu+\displaystyle\int_{\mathbb{S}^n}\frac{2n}{f^{n-1}}d\mu_{\mathbb{S}^n}}{\left(\displaystyle\int_{\Sigma}\frac{1}{f^{2n}}d\mu\right)^{\frac{n-1}{n}}}
\end{align*}
Using $r\leq f\leq R$ on $\mathbb{S}^n$, we can estimate
\begin{align*}
    Q_1(\widetilde{\Sigma})&\leq\frac{-\displaystyle\frac{1}{R^{2(n-1)}}\int_{\Sigma}Hd\mu+\displaystyle\frac{1}{r^{n-1}}\int_{\mathbb{S}^n}2nd\mu_{\mathbb{S}^n}}{\left(\displaystyle\frac{1}{R^{2n}}\int_{\Sigma}d\mu\right)^{\frac{n-1}{n}}} \\
    %%%
    &=\frac{-\displaystyle\frac{1}{R^{2(n-1)}}\int_{\Sigma}Hd\mu+\displaystyle\frac{2n|\mathbb{S}^n|}{r^{n-1}}}{\displaystyle\frac{1}{R^{2(n-1)}}|\Sigma|^{\frac{n-1}{n}}} \\
    %%%
    &=-\frac{1}{|\Sigma|^{\frac{n-1}{n}}}\int_{\Sigma}Hd\mu+\left(\frac{R}{r}\right)^{n-1}R^{n-1}\cdot\frac{2n|\mathbb{S}^n|}{|\Sigma|^{\frac{n-1}{n}}} \\
    %%%
    \therefore\qquad Q_1(\Sigma)+Q_1(\widetilde{\Sigma})&\leq\left(\frac{R}{r}\right)^{n-1}R^{n-1}\cdot\frac{2n|\mathbb{S}^n|}{|\Sigma|^{\frac{n-1}{n}}}
\end{align*}
Next, note that the inversion of $\widetilde{\Sigma}$ is just $\Sigma$. Thus, using $\frac{1}{R}\leq\frac{1}{f}\leq\frac{1}{r}$, we have 
\begin{align*}
    Q_1(\widetilde{\Sigma})+Q_1(\Sigma) &\leq\left(\frac{1/r}{1/R}\right)^{n-1}\left(\frac{1}{r}\right)^{n-1}\cdot\frac{2n|\mathbb{S}^n|}{|\widetilde{\Sigma}|^{\frac{n-1}{n}}} \\
    %%%
    &=\left(\frac{R}{r}\right)^{n-1}\left(\frac{1}{r}\right)^{n-1}\cdot\frac{2n|\mathbb{S}^n|}{|\widetilde{\Sigma}|^{\frac{n-1}{n}}}
\end{align*}
Therefore, 
\begin{align*}
    \left(Q_1(\Sigma)+Q_1(\widetilde{\Sigma})\right)^2&\leq\left(\frac{R}{r}\right)^{n-1}R^{n-1}\cdot\frac{2n|\mathbb{S}^n|}{|\Sigma|^{\frac{n-1}{n}}}\cdot\left(\frac{R}{r}\right)^{n-1}\left(\frac{1}{r}\right)^{n-1}\cdot\frac{2n|\mathbb{S}^n|}{|\widetilde{\Sigma}|^{\frac{n-1}{n}}} \\ 
    %%%
    &=\left(\frac{R}{r}\right)^{3(n-1)}\frac{\left(2n|\mathbb{S}^n|\right)^2}{\left(|\Sigma|\cdot|\widetilde{\Sigma}|\right)^{\frac{n-1}{n}}}.
\end{align*}
Hence, we have proved one side of the inequality:
\begin{align*}
    Q_1(\Sigma)+Q_1(\widetilde{\Sigma})\leq\left(\frac{R}{r}\right)^{\frac{3(n-1)}{2}}\frac{2n|\mathbb{S}^n|}{\left(|\Sigma|\cdot|\widetilde{\Sigma}|\right)^{\frac{n-1}{2n}}}
\end{align*}
The other inequality is proved similarly. \bigskip 

Note that the only occasions where we have inequalities are when we estimate $f$ and $1/f$ using $r\leq f\leq R$. Thus, the inequalities above become equalities if and only if $r=f=R$ on $\mathbb{S}^n$; that is, $f$ is a constant, and this happens if and only if $\Sigma$ is a round sphere. 
\end{proof}

\begin{remark}
The equality case of Theorem \ref{thm:inequality} can give an alternate proof of a special case of Theorem \ref{thm:StarShaped} when the conformal Killing field $V$ corresponds to rescaling and inversion. First recall that the classical Minkowski's inequality (see e.g. the IMCF proof by Guan-Li \cite{GL}), which asserts that $Q_1(\Sigma), Q_1(\widetilde{\Sigma}) \geq n\abs{\mathbb{S}^n}$ and equality holds if and only if $\Sigma$ is a round sphere. By Gerhardt \cite{G} and Urbas \cite{U}, the flow $(\partial_t F)^\perp = -\frac{1}{\rho}{v}$ converges after rescaling to a sphere with a fixed radius (see e.g. \cite[P.313]{G}). Then, by $\frac{R}{r} \to 1$ as $t \to +\infty$ under the rescaled flow, the inequality \eqref{eq:inequality} becomes an equality as at time infinity:
\[2n\abs{\mathbb{S}^n} = \overline{Q}(\Sigma).\]
Note that the quantity $\overline{Q}(\Sigma)$ is unchanged if the flow evolves by rescaling and inversion. Therefore, at any time $Q_1(\Sigma)$ and $Q_1(\widetilde{\Sigma})$ both achieves the equality case of the Minkowski's inequality. Hence, it can also conclude that $\Sigma$ is a round sphere.
\end{remark}

\bibliographystyle{amsplain}
\bibliography{citations}

\end{document}